\documentclass[a4paper,12pt,reqno]{amsart}

\usepackage{color}
\usepackage{amsmath}
\usepackage{amssymb}
\usepackage{amsfonts}
\usepackage{graphicx}
\usepackage{mathtools}
\usepackage[colorlinks]{hyperref}
\renewcommand\eqref[1]{(\ref{#1})} 
\usepackage{booktabs}
\usepackage{lscape}
\usepackage{mathrsfs}
\usepackage{comment}
\usepackage{caption}
\graphicspath{ {images/} }
\newcommand{\Q}{\mathbb{Q}}
\newcommand{\R}{\mathbb{R}}

\title[Boundedness of spectral multipliers]{Boundedness of spectral multipliers on locally compact groups and applications}

\author[S. G\'omez Cobos]{Santiago G\'omez Cobos}
\address{
	Santiago G\'omez Cobos:
	\endgraf
	Department of Mathematics: Analysis, Logic and Discrete Mathematics
	\endgraf
	Ghent University, Krijgslaan 281, Building S8, B 9000 Ghent
	\endgraf
	Belgium
	\endgraf
	{\it E-mail address} {\rm davidsantiago.gomezcobos@ugent.be}}

\author[J. E. Restrepo]{Joel E. Restrepo}
\address{
	Joel E. Restrepo:
	\endgraf
	Department of Mathematics
	\endgraf
Center for Research and Advanced Studies 
\endgraf
Av. IPN 2508, 07360, CDMX 
    \endgraf
	Mexico
	\endgraf
	{\it E-mail address} {\rm joel.restrepo@cinvestav.mx}}

\author[M. Ruzhansky]{Michael Ruzhansky}
\address{
	Michael Ruzhansky:
	\endgraf
	Department of Mathematics: Analysis, Logic and Discrete Mathematics
	\endgraf
	Ghent University, Krijgslaan 281, Building S8, B 9000 Ghent
	\endgraf
	Belgium
	\endgraf
	and
	\endgraf
    School of Mathematical Sciences
    \endgraf
    Queen Mary University of London
    \endgraf
    United Kingdom
    \endgraf
	{\it E-mail address} {\rm michael.ruzhansky@ugent.be}}


\newtheoremstyle{theorem}
{10pt}          
{10pt}  
{\sl}  
{\parindent}     
{\bf}  
{. }    
{ }    
{}     
\theoremstyle{theorem}

\numberwithin{equation}{section}
\theoremstyle{plain}
\newtheorem{thm}{Theorem}[section]

\newtheorem{lem}[thm]{Lemma}

\theoremstyle{definition}
\newtheorem{defn}[thm]{Definition}
\newtheorem{rem}[thm]{Remark}

\newtheoremstyle{defi}
{10pt}          
{10pt}  
{\rm}  
{\parindent}     
{\bf}  
{. }    
{ }    
{}     
\theoremstyle{defi}

\begin{document}
\begin{abstract}
We prove that the noncommutative Lorentz norm (associated to a semifinite von Neumann algebra) of a propagator of the form $\varphi(|\mathscr{L}|)$ can be estimated if the modulus of the Borel function $\varphi$ is bounded by a continuous positive monotonically decreasing function that vanishes at infinity $\psi$. As a consequence, we obtain the $L^p-L^q$ $(1<p\leqslant 2\leqslant q<+\infty)$ norm estimates for the solutions of heat, wave, and Schr\"odinger type equations (new in this setting) on a locally compact separable unimodular group $G$ by using a non-local integro-differential operator in time and any positive left invariant operator (maybe unbounded and with discrete or continuous spectrum) on $G$. We also provide asymptotic estimates (large-time behavior) for the solutions, which in some cases can be claimed to be sharp. Illustrative examples are given for several groups. 
\end{abstract}
	\maketitle

\section{Introduction}
The regularity of integro-differential equations has been in the sights of mathematicians for a long time not only for its wide range of applications but also for its abstract relevance \cite{ade,pruss,tricomi}. From the abstract point of view one can quickly identify that there are two promising ways to extend this kind of studies, on one hand one can try to study abstract equations on simpler spaces (the basic model is $\R^n$, see e.g. \cite{bre,hor,lit,stri-intro}); or on the other hand one can work with some simpler equations on very general spaces ($C^*$-algebras, symmetric spaces, groups, etc, see e.g. \cite{quantum, cow, giga}). In \cite{RR2020}, Akylzhanov and the third author went one step further in the direction of the second option mentioned above by obtaining regularity in a quite general set up, considering linear evolution equations on locally compact unimodular groups; to describe the generality, the type of groups which can be considered is very wide. Some of them are compact, semisimple, exponential, nilpotent, some solvable ones, real algebraic, and many more. Moreover, the operators which can be involved are very general as well, and in particular one can use the following relevant ones: Laplacians, sublaplacians, (non)-Rockland operators, weighted subcoercive operators, etc. This was done through an application of a spectral multiplier boundedness theorem, which was obtained mainly due to the application of von Neumann algebras theory \cite{von,[46]} and the generalized non-commutative Lorentz spaces \cite{[51]}. Using their result \cite[Theorem 6.1]{RR2020} one can translate the regularity problem to the study of boundedness of propagators of the form $\varphi(|\mathscr{L}|\footnote{We denote by $L=U|L|$ its polar decomposition.})$ for $\varphi$ a continuous positive monotonically decreasing function vanishing at infinity, and if we are provided good asymptotic behavior of the traces of the spectral projections $E_{(s,+\infty)}(|\mathscr{L}|)$ we end up just imposing properties of the function $\varphi$. In 
\cite{RR2020}, the main examples were functions $e^{-tx}$ and $1/(1+x)^{\gamma}$ for $t,x,\gamma>0,$ i.e. boundedness of solutions of the $\mathscr{L}$-heat equation and Sobolev embeddings, respectively.  Obviously, one cannot expect to be able to use any function $\varphi$, and in particular it does not work for $\varphi(x)=e^{it\sqrt{x}}$ with $t,x>0$, which represents the propagator for the wave equation, hence it is not covered by these results as well as the Schrödinger one. Nevertheless, on a compact/graded Lie group in \cite{wagner,joel}, it was shown that for the non-local in time $\mathscr{L}$-heat equation ($\mathscr{L}$-heat type) one can apply the previous reasoning since in those cases the propagator is related to the well-known Mittag-Leffler functions \cite{mittag}. 

Having this in mind, the authors were motivated to extend the class of equations which can be studied in this way, and we managed to obtain the following estimation of the Lorentz norms of a general propagator:  
\begin{thm}\label{additional}
    Let $L$ be a closed (maybe unbounded) operator affiliated\footnote{See the preliminaries for the definition.} with a semifinite von Neuman algebra $M$. Let $\phi:[0,+\infty)\to\mathbb{C}$ be a Borel measurable function. Suppose also that $\psi$ is a monotonically decreasing continuous function on $[0,+\infty)$ such that $\psi(0)=1$, $\lim_{v\to+\infty}\psi(v)=0$ and $|\phi(v)|\leqslant \psi(v)$ for all $v\in[0,+\infty).$ Then for every $1\leqslant r<+\infty$ we have the inequality 
\[
\|\phi(|L|)\|_{L^{r, \infty}(M)}\leqslant \sup_{v>0}\psi(v)\big[\tau(E_{(0,v)}(|L|))\big]^{\frac{1}{r}}.
\] 
\end{thm}
This improvement allows us, at this stage, to treat the $\mathscr{L}$-wave type equation, by this we mean the wave equation but replacing the classical time derivative by the non-local operator $^{C}\partial_{t}^{\beta}w(t,x)$, the so-called Dzhrbashyan--Caputo fractional derivative, defined by $^{C}\partial_{t}^{\beta}w(t,x)=\prescript{RL}{0}I^{2-\beta}\partial_t^{(2)} w(t,x),$ where $\beta\in(1,2)$ and $\prescript{RL}{0}I^{\beta}w(t,x)=\frac1{\Gamma(\beta)}\int_0^t (t-s)^{\beta-1}w(s,x)\,\mathrm{d}s,$ is the Riemann-Liouville fractional integral of order $\beta>0.$ For $\beta=2$ we have the classical partial derivative in time, i.e. $^{C}\partial_{t}^{\beta}w(t,x)=\prescript{RL}{0}I^{0}\partial_t^{(2)} w(t,x)=\partial_t^{(2)} w(t,x)$ since $\prescript{RL}{0}I^{0}$ acts like the identity operator. Notice also that for the range $0<\beta<1$ we have $^{C}\partial_{t}^{\beta}w(t,x)=\prescript{RL}{0}I^{1-\beta}\partial_t w(t,x),$ and $\beta=1$ coincides with the classical partial derivative in time $^{C}\partial_{t}^{\beta}w(t,x)=\partial_t w(t,x).$ Also, the theorem above can be applied for Schr\"odinger type equations, specifically 
\[i\prescript{C}{}\partial_{t}^{\beta}w(t,x)+\mathscr{L}w(t,x)=0,\quad \text{\rm for}\quad 0<\beta<1.\] 
These applications are complemented by the study of the $\mathscr{L}$-heat type equation, i.e.
\[
^{C}\partial_{t}^{\beta}w(t,x)+\mathscr{L}w(t,x)=0,\quad\text{\rm for}\quad 0<\beta<1. 
\]
Notice that for the classical heat equation, the result was already obtained in \cite{RR2020}. 

For the reader's convenience we now plot the propagators (functions) for the wave type equation, which are given in terms of the two parametric Mittag-Leffler function $E_{\alpha,\delta}(-x)$ for $x>0$ (see formula \eqref{bimittag}), so that one can visualize the intrinsic idea of our applications. The functions of the propagators are of the form $E_{\alpha,1}(-x)$ and $E_{\alpha,2}(-x)$ for $1<\alpha<2$, see Section \ref{main-aplications} for full details. The idea is to show a clear picture of the oscillating behavior, for a particular $\alpha$, of these type of functions which are always bounded uniformly by $C/(1+x)$ for some positive constant $C$ (see formula \eqref{uniform-estimate}): 
\begin{center}
    \begin{figure}[ht]
        \centering
        \includegraphics[scale=0.41]{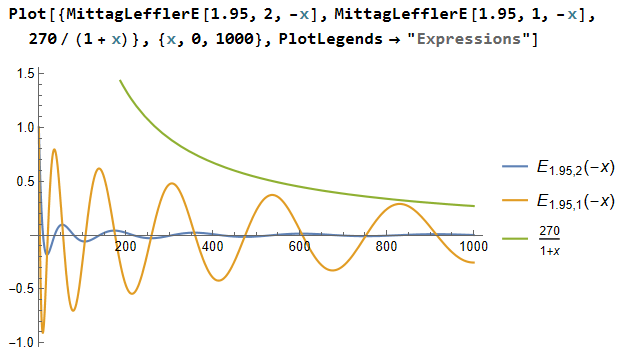}
        \caption{Wave propagators functions for $\alpha=1.95$ bounded uniformly.}
        \label{graficaBOUNDED}
    \end{figure}
\end{center}
In Figure \ref{graficaNOT}, one can also realize that for $\alpha=2$ we have an oscillating behaviour like cosine or sine, and does not vanish at infinity:
\begin{center}
    \begin{figure}[ht]
        \centering
\captionsetup{justification=centering}
        \includegraphics[scale=0.42]{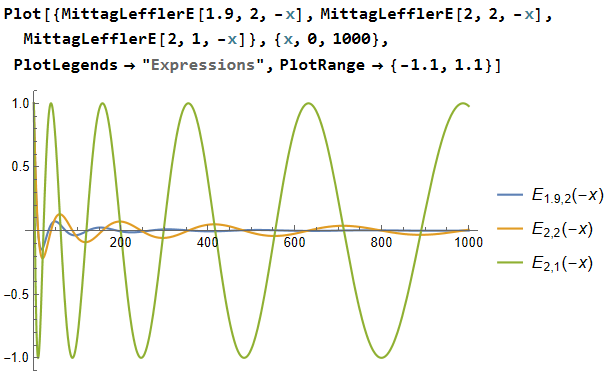}
        \caption{Classical wave propagator function (green) is not uniformily bounded by a decreasing vanishing at infinity function.}
        \label{graficaNOT}
    \end{figure}
\end{center}
Also, we give a complementary perspective of the orders and behaviour of some other type of propagators associated with Mitagg-Leffler functions which can be bounded uniformly:
\begin{center}
    \begin{figure}[ht]
        \centering
        \includegraphics[scale=0.41]{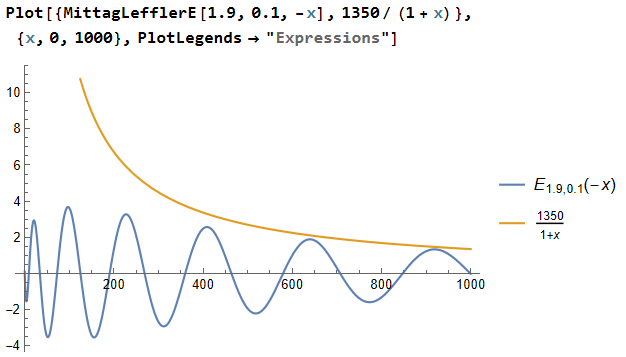}
        \caption{Mittag-Leffler functions with one small parameter.}
        \label{grafica}
    \end{figure}
\end{center}
   Let us now summarize the applications concerning the $L^p-L^q$ $(1<p\leqslant 2\leqslant q<+\infty)$ norm estimates for $\mathscr{L}$-heat-wave-Schr\"odinger type equations. These results are new in this context and open a door to studying different estimates of other type of equations which, in principle, can not be calculated in the classical case but perhaps could be carried out in the set up of non-local operators. Also, we discuss existence, uniqueness and time-decay behaviour of such equations. Full details are given in Section \ref{main-aplications}.

Now we write explicitly the type of equations for which one can apply the main theorem above. 

\medskip{\bf $\mathscr{L}$-heat-wave type equation}: 
\begin{align}
\int_0^t \frac{(t-s)^{n-\beta-1}}{\Gamma(n-\beta)}\partial_{s}^{(n)}w(s,x)\,\mathrm{d}s+\mathscr{L}w(t,x)&=0,\quad t>0,\,\, x\in G, \label{heat-wave-intro}\\
w(t,x)|_{_{_{t=0}}}&=w_0(x),\quad w_0\in L^p(G), \label{111} \\
\partial_t w(t,x)|_{_{_{t=0}}}&=w_1(x),\quad w_1\in L^p(G), \label{222}
\end{align}
where $0<\beta<2$ $(\beta\neq1)$, $n=\lfloor\beta\rfloor+1$ ($\lfloor\cdot\rfloor$ is the floor function), $1<p\leqslant 2,$ $\mathscr{L}$ is any positive linear left invariant operator on $G$ (a locally compact separable unimodular group). If $0<\beta<1$ we just consider the first initial condition $w_0$. While, for $1<\beta<2$, we use both initial conditions $w_0,w_1$. We exclude the value $\beta=1$ from consideration since this case is already known (\cite{RR2020}).

\medskip We also study the {\bf $\mathscr{L}$-Schr\"odinger type equation}: 
\begin{align}
i\int_0^t \frac{(t-s)^{-\beta}}{\Gamma(1-\beta)}\partial_{s}w(s,x)\,\mathrm{d}s+\mathscr{L}w(t,x)&=0,\quad t>0,\,\, x\in G,\,\,0<\beta<1, \label{schrodinger-intro}\\
w(t,x)|_{_{_{t=0}}}&=w_0(x),\quad w_0\in L^p(G). \label{333} 
\end{align}

\begin{thm}
Let $G$ be a locally compact separable unimodular group and let $1<p\leqslant 2\leqslant q<+\infty$. Let $\mathscr{L}$ be any positive left invariant operator on $G$ (maybe unbounded).  
\begin{itemize}
    \item If $0<\beta<1$, $w_0\in L^p(G)$ and  
    \[
\sup_{t>0}\sup_{s>0}\big[\tau\big(E_{(0,s)}(\mathscr{L})\big)\big]^{\frac{1}{p}-\frac{1}{q}}E_\beta(-t^{\beta}s)<+\infty,
\]
then there exists a unique solution $w\in\mathcal{C}\big((0,+\infty) ;L^q(G)\big)$ for the {\bf $\mathscr{L}$-heat type equation} \eqref{heat-wave-intro}-\eqref{111} given by 
\[
w(t,x)=E_\beta(-t^{\beta}\mathscr{L})w_0(x),\quad t>0,\,\,x\in G.
\]
\item If $1<\beta<2$, $w_0,w_1\in L^p(G)$ and 
\begin{equation}\label{need-2-intro}
\sup_{t>0}\sup_{s>0}\frac{\big[\tau\big(E_{(0,s)}(\mathscr{L})\big)\big]^{\frac{1}{p}-\frac{1}{q}}}{1+t^\beta s}<+\infty,   
\end{equation}
then there exists a unique solution $w\in\mathcal{C}\big((0,+\infty) ;L^q(G)\big)$ for the {\bf $\mathscr{L}$-wave type equation} \eqref{heat-wave-intro}-\eqref{222} given by
\[
w(t,x)=E_\beta(-t^{\beta}\mathscr{L})w_0(x)+\prescript{RL}{0}I^{1}_t E_{\beta}(-t^{\beta}\mathscr{L})w_1(x),\quad t>0,\,\,x\in G.
\]
\item If $w_0\in L^p(G)$ and the condition \eqref{need-2-intro} holds with $0<\beta<1$, then there exists a unique solution $w\in\mathcal{C}\big((0,+\infty) ;L^q(G)\big)$ for the {\bf $\mathscr{L}$-Schr\"odinger type equation} \eqref{schrodinger-intro}-\eqref{333} given by 
\[
w(t,x)=E_\beta(it^{\beta}\mathscr{L})w_0(x),\quad t>0,\,\,x\in G.
\]
\end{itemize}

In particular, if    
\[
\tau\big(E_{(0,s)}(\mathscr{L})\big)\lesssim s^{\lambda},\quad s\to+\infty,\quad\text{for some}\quad \lambda>0,
\]
we obtain the following time decay rate for the solutions of $\mathscr{L}$-heat type equation \eqref{heat-wave-intro}-\eqref{111} and $\mathscr{L}$-Sch\"odinger type equation \eqref{schrodinger-intro}-\eqref{333}:  
\[
\|w(t,\cdot)\|_{L^q(G)}\leqslant C_{\beta  ,\lambda,p,q}t^{-\beta\lambda\left(\frac{1}{p}-\frac{1}{q}\right)}\|w_0\|_{L^p(G)},   
\]
and also for the $\mathscr{L}$-wave type equation \eqref{heat-wave-intro}-\eqref{222}:
\[
\|w(t,\cdot)\|_{L^q(G)}\leqslant  C_{\beta,\lambda,p,q}t^{-\beta\lambda\left(\frac{1}{p}-\frac{1}{q}\right)} \big(\|w_0\|_{L^p(G)}+t\|w_1\|_{L^p(G)}\big),   
\]
whenever $\frac{1}{\lambda}\geqslant\frac{1}{p}-\frac{1}{q}.$
\end{thm}

Preceding results are used to give different examples by choosing concrete operators and groups. Moreover, in some occasions, we are able to claim the sharpness of the time-decay rate. See Section \ref{examples-section}.   

\section{Preliminary results}\label{preli}

In this section we recall definitions and some necessary results on integro-differential operators, abstract fractional differential equations and von Neumann algebras, which will appear frequently in this paper. 

\subsection{Abstract integro-differential equations}\label{abstractsection}
Here we recall some concepts with respect to the solution of an evolution integral equation. General and fundamental aspects of this theory can be found in e.g. the books \cite{pruss,tricomi}. For a more specialized treatment of abstract non-local (fractional) differential equations in this framework we can refer to the works \cite{thesis2001,thesis,miller}.  

\subsubsection{Integro-differential operators}

As usual, the Riemann--Liouville fractional integral of order $\beta>0$ (\cite{samko}) is defined by
\[
\prescript{RL}{a}I^{\beta}f(t)=\frac1{\Gamma(\beta)}\int_a^t (t-s)^{\beta-1}f(s)\,\mathrm{d}s,\qquad f\in L^1(a,T),
\]
where $L^1(a,T)$ is the Lebesgue integrable space on $(a,T).$  

\medskip Now we recall the well-known Dzhrbashyan--Caputo fractional derivative:
\begin{equation}\label{caputo-alternative-uso}
\prescript{C}{a}D^{\beta}f(t)=\prescript{RL}{a}I^{n-\beta}f^{(n)}(t),\qquad f\in AC^n[a,T],\quad n=\lfloor\beta\rfloor+1,
\end{equation}
where $AC^n[a,T]$ is the set of functions such that $f^{(n-1)}$ exists and is absolutely continuous on $[a,T].$ 

Below we use another equivalent representation for operator \eqref{caputo-alternative-uso}. The new form of this operator is useful in applications since it can be rewritten utilizing the initial conditions. For this reason, we first need to recall some suitable Sobolev type spaces \cite[Appendix]{sobolev}:  
\[
W^{m,p}(I;X):=\left\{g\bigg/\exists\phi\in L^{p}(I;X);\quad g(t)=\sum_{j=0}^{m-1}a_j \frac{t^{j}}{j!}+\frac{t^{m-1}}{(m-1)!}*\phi(t),\quad t\in I\right\},
\]
where $m\in\mathbb{N}$, $1\leqslant p<+\infty$ and $I$ is an interval in $\mathbb{R}.$ One can see that $a_j=g^{(j)}(0)$ and $\phi(t)=g^{(j)}(t).$ So, if $f\in C^{n-1}(I)$ and $h_{n-\beta}*f\in W^{n,1}(I)$, where $n=\lfloor\beta\rfloor+1$ and
\[
h_\beta(t):=\left\{
\begin{array}{rccl}
\frac{t^{\beta-1}}{\Gamma(\beta)}, & t>0, \\
0,& t\leqslant0,
\end{array}
\right.
\]
then we get
\begin{equation}\label{otra-ini}
\prescript{C}{a}D^{\beta}f(t)=\prescript{RL}{a}D^{\beta}\left(f(t)-\sum_{k=0}^{n-1}\frac{f^{(k)}(a)}{k!}(t-a)^k\right),
\end{equation}
where $\prescript{RL}{a}D^{\beta}f(t)=D^{n}\prescript{RL}{a}I^{n-\beta}f(t)$ is the Riemann--Liouville fractional derivative. For more details, see e.g. \cite[P. 10]{thesis2001}.

\subsubsection{Integro-differential equations}\label{fde} Below we always consider the linear operator $\mathscr{L}$ to be closed and densely defined in a Banach space $X$. We study the following Cauchy problem by means of the Dzhrbashyan-Caputo fractional derivative: 
\begin{equation}\label{abstracte}
\prescript{C}{0}\partial^{\beta}_t w(t)=\mathscr{L}w(t),\quad t>0, 
\end{equation}
under the initial conditions
\begin{equation}\label{ae-conditions}
\frac{\partial^{j}w(t)}{\partial t^{j}}\Big|_{_{_{t=0+}}}=w_j\in X,\quad j=0,1,\ldots,n-1,    
\end{equation}
where $\beta>0$ and $n=1+\lfloor \beta\rfloor$.
\begin{defn}
Let $w$ be a function of the space $C(\mathbb{R}^{+},X)$. It is said that $w$ is a strong solution of equation \eqref{abstracte}-\eqref{ae-conditions} if $w\in C(\mathbb{R}^{+},X)\cap C^{n-1}(\mathbb{R}^{+},X)$, 
\[
\prescript{RL}{0}I^{n-\beta}\left(w(t)-\sum_{j=0}^{n-1}\frac{\partial^{j}w(t)}{\partial t^{j}}\Big|_{_{_{t=0+}}}\frac{t^j}{j!}\right)\in C^{n-1}(\mathbb{R}^{+},X)
\]
and equation \eqref{abstracte}-\eqref{ae-conditions} is satisfied on $(\mathbb{R}^+,X).$
\end{defn}

\begin{defn}
The fractional Cauchy problem \eqref{abstracte}-\eqref{ae-conditions} is said to be well-posed if the following two conditions are satisfied: 
\begin{enumerate}
    \item For any $w_j\in\mathcal{D}(\mathscr{L})$, $j=0,\ldots,n-1$, there exists a unique strong solution $w$ of equation \eqref{abstracte}-\eqref{ae-conditions}.
    \item For all sequences  $\{w_{k,m}\}_{m\geqslant1}\subset\mathcal{D}(\mathscr{L})$ $(k=0,\ldots,n-1)$, such that $w_{k,m}\to0$ as  $m\to+\infty$, imply $w_m\to0$ as $m\to+\infty$ in $X$, uniformly on compact intervals, where $w_m$ is the unique strong solution of equation \eqref{abstracte}-\eqref{ae-conditions} with initial conditions $w_{0,m},w_{1,m},\dots,w_{n-1,m}$.  
\end{enumerate} 
\end{defn}

We now show that the study of the solution of equation \eqref{abstracte}-\eqref{ae-conditions} can be reduced to the investigation of equation \eqref{abstracte} under the following initial conditions: 
\begin{equation}\label{abstractezero}
w_0(t)\Big|_{_{_{t=0+}}}=w_0,\quad \frac{\partial^{j}w(t)}{\partial t^{j}}\Big|_{_{_{t=0+}}}=0,\quad j=1,\ldots,n-1. 
\end{equation}
First of all, by \cite[Formula (1.21)]{thesis2001}, one can see that equation \eqref{abstracte},\eqref{abstractezero} is well-posed if and only if the following Volterra integral equation 
\begin{equation}\label{abstracti}
w(t)=w_0(t)+\prescript{RL}{0}{I}^{\beta}\big(\mathscr{L}w(t)\big),    
\end{equation}
is well-posed in the sense of \cite[Def. 1.2]{pruss}. Hence, we are prepared to define the notion of a solution operator (we always assume exponentially bounded) of equation \eqref{abstracte},\eqref{abstractezero} by means of its equivalent integral equation \eqref{abstracti}.  

\begin{defn}
A family $\{E_\beta(-t^{\beta}\mathscr{L})\}_{t\geqslant0}\subset \mathcal{B}(X)$ is called a solution operator of equation \eqref{abstracte},\eqref{abstractezero} if the following conditions are satisfied:
\begin{enumerate}
    \item $E_\beta(-t^{\beta}\mathscr{L})$ is strongly continuous for $t\geqslant0$ and $E_\beta(0)=I;$
    \item $E_\beta(-t^{\beta}\mathscr{L})\mathcal{D}(\mathscr{L})\subset\mathcal{D}(\mathscr{L})$ and $\mathscr{L}E_\beta(-t^{\beta}\mathscr{L})w=E_\beta(-t^{\beta}\mathscr{L})\mathscr{L}w$ for any $w\in\mathcal{D}(\mathscr{L})$, $t\geqslant0;$
    \item $E_\beta(-t^{\beta}\mathscr{L})w$ is a solution of equation \eqref{abstracti} for any $w\in\mathcal{D}(\mathscr{L})$, $t\geqslant0.$
\end{enumerate}
\end{defn}
By using the same terminology of \cite{pruss}, equation \eqref{abstracte},\eqref{abstractezero} is well-posed if and only if it has a solution operator. Therefore, if $E_{\beta}(-t^{\beta}\mathscr{L})$ is a solution operator of equation \eqref{abstracte},\eqref{abstractezero} then the general equation \eqref{abstracte}-\eqref{ae-conditions} is uniquely solvable by
\[
w(t)=\sum_{j=0}^{n-1}\prescript{RL}{0}I^{j}\big(E_\beta(-t^{\beta}\mathscr{L})\big)w_j,
\]
where $w_j\in\mathcal{D}(\mathscr{L})$, $j=0,1,\ldots,n-1.$ So, it is well-posed. This means that is enough to study equation \eqref{abstracte},\eqref{abstractezero} to know the solution of the general case \eqref{abstracte}-\eqref{ae-conditions}.

The solution operator can be written in an integral form as \cite[Chapters 2 and 3]{thesis2001} (see also \cite[Chapter 2]{pruss} or \cite[Theorem 2.41]{thesis})
\[
E_\beta(-t^{\beta}\mathscr{L})=\frac{1}{2\pi i}\int_{H}e^{\gamma t}\gamma^{\beta-1}(\gamma^{\beta}+\mathscr{L})^{-1}{\rm d}\gamma,\quad t\geqslant0,\quad \beta\in(0,2),
\]
where $H$ is the Hankel path of \cite[Formula (2.5)]{thesis} and it is contained in the resolvent set $\rho(-\mathscr{L})$. The above notation is consistent (convenient) since the integral representation of the Mittag-Leffler function  $E_{\alpha}(z)=\sum_{k=0}^{+\infty}\frac{z^k}{\Gamma(\alpha k+1)}$ for $z\in\mathbb{C}$ and $\Re(\alpha)>0$ is given by   
\[
E_{\alpha}(z)=\frac{1}{2\pi i}\int_{\mathcal{H}}e^{\gamma}\gamma^{\alpha-1}(\gamma^{\alpha}-z)^{-1}{\rm d}\gamma,
\]
where $\mathcal{H}$ is a suitable Hankel path \cite[Formula (3.4.12)]{mittag}.   

In this paper, the two-parametric Mittag-Leffler function will appear frequently:
\begin{equation}\label{bimittag}
E_{\alpha,\delta}(z)=\sum_{k=0}^{+\infty} \frac{z^k}{\Gamma(\alpha k+\delta)},\quad z,\delta\in\mathbb{C},\quad \text{\rm Re}(\alpha)>0,
\end{equation}
which is absolutely and locally uniformly convergent for the given parameters (\cite{mittag}).

Also, we need the following particular estimate from  \cite[Theorem 1.6]{page 35}: 
\begin{equation}\label{uniform-estimate}
|E_{\alpha,\delta}(z)|\leqslant \frac{C}{1+|z|},\quad z\in\mathbb{C},\,\,\delta\in\mathbb{R},\,\, \alpha<2,
\end{equation}
where $\mu\leqslant |\arg(z)|\leqslant \pi$, $\pi\alpha/2<\mu<\min\{\pi,\pi \alpha\}$ and $C$ is a positive constant.

\subsection{Von Neumann algebras} The starting point of this algebras can be found in the well-known papers of von Neumann in \cite{[46],[47]}, where one builds the mathematical foundations for studying quantum mechanics. Let $\mathfrak{L}(\mathcal{H})$ be the set of linear operators defined on a Hilbert space $\mathcal{H}$. Notice that the concept of $\tau$-measurability on a von Neumann algebra $M$ and the technique of spectral projections allow us to approximate unbounded operators by bounded ones. 

For our study we just need to set $M$ to be the right group von Neumann algebra $VN_R(G)$ with $G$ being a locally compact separable unimodular group. The latter assumption gives the possibility to use the notion of noncommutative Lorentz spaces on $M$ as given in \cite{[51]}. For more specific details on von Neumann algebras we refer \cite{von,von2}.  

It is important to recall that the group von Neumann algebra $VN_R(G)$ is generated by all the right actions of $G$ on $L^2(G)$ ($\pi_R(g)f(x)=f(xg)$ with $g\in G$), which means that $VN_R(G)=\{\pi_R(G)\}^{!!}_{g\in G}$, where $!!$ is the bicommutant of the subalgebra $\{\pi_R(g)\}_{g\in G}\subset \mathfrak{L}(L^2(G))$. The latter result is a consequence of the fact \cite{von}: $VN_R(G)^{!}=VN_L(G)$ and $VN_L(G)^{!}=VN_R(G)$, where the symbol $!$ represents the commutant of the group. 

Before recalling the noncommutative Lorentz spaces we need the following definitions.  

\begin{defn}
Let $M\subset\mathfrak{L}(\mathcal{H})$ be a semifinite von Neumann algebra acting over the Hilbert space $\mathcal{H}$ with a trace $\tau$. A linear operator $L$ (maybe unbounded) is said to be affiliated with $M$, if it commutes with the elements of the commutant $M^!$ of $M$, which means $LV=VL$ for all $V\in M^!.$
\end{defn}

We highlight that in this paper we will consider operators affiliated with $VN_R(G)$ (see \cite[Def. 2.1]{RR2020}), which are precisely those ones who are left invariant on $G$ \cite[Remark 2.17]{RR2020}.

\begin{rem}
If $L$ is an affiliated bounded operator with $M$, then $L\in M$ by the double commutant theorem. 
\end{rem}

So, the next definitions are given for the above type of operators.   

\begin{defn}
Let $M$ be a von Neumann algebra. A trace of the positive part $M_+=\{L\in M: L^*=L>0\}$ of $M$ is a functional defined on $M_+$, taking non-negative, possibly infinite, real values, with the following properties:
\begin{enumerate}
    \item If $L\in M_+$ and $T\in M_+$ then $\tau(L+T)=\tau(L)+\tau(T)$;
    \item If $L\in M_+$ and $\gamma\in\mathbb{R}^{+}$ then $\tau(\gamma L)=\gamma\tau(L)$ (with $0\cdot+\infty=0$);
    \item If $L\in M_+$ and $U$ is an unitary operator on $M$ then $\tau(ULU^{-1})=\tau(L).$
\end{enumerate}
\end{defn}

A trace $\tau$ is faithful (or exact) if $\tau(L)=0$ $(L\in M_+)$ implies that $L=0$. We have that $\tau$ is finite if $\tau(L)<+\infty$ for all $L\in M_+$. We have that $\tau$ is semifinite if, for each $L\in M_+$, $\tau(L)$ is the supremum of the numbers $\tau(T)$ over those $T\in M_+$ such that $L\leqslant T$ and $\tau(T)<+\infty.$  

\begin{defn}
A closeable operator $L$ (maybe unbounded) is called  $\tau$-measureable if for each $\epsilon>0$ there exists a projection $P$ in $M$ such that $P(\mathcal{H})\subset D(L)$ and $\tau(I-P)\leqslant \epsilon$, where $D(L)$ is the domain of $L$ in $\mathcal{H}$ and $M\subset\mathfrak{L}(\mathcal{H})$. We denote by $S(M)$ the set of all $\tau$-measurable operators.   
\end{defn}

\begin{defn}\label{generalTValues}
Let us take an operator $L\in S(M)$ and let $L = U|L|$ be its polar decomposition. We define the distribution function by $d_\gamma(L):=\tau\big(E_{(\gamma,+\infty)}(|L|)\big)$ for $\gamma\geqslant0$, where $E_{(\gamma,+\infty)}(|L|)$ is the spectral projection of $|L|$ over the interval $(\gamma,+\infty).$ Also, for any $t>0$, we define the generalized $t-$th singular numbers as 
\[
\mu_t(L):=\inf\{\gamma\geqslant0:\,d_\gamma(L)\leqslant t\}.
\]
\end{defn}
For more details and properties of the distribution function and generalized singular numbers, we recommend \cite{pacific}. 

Below we recall the noncommutative Lorentz spaces associated with a semifinite von Neumann algebra $M$, which are a noncommutative extension of the classical Lorentz spaces \cite{[51]}. After that, we provide a useful result that will help us to prove the main result in the next section.

\begin{defn}
We denote by $L^{p,q}(M)$ $(1\leqslant p<+\infty,\,1\leqslant q<+\infty)$ the set of all operators $L\in S(M)$ such that
\begin{align*}
\|L\|_{L^{p,q}(M)}=\left(\int_0^{+\infty}\big(t^{1/p}\mu_t(L)\big)^q \frac{{\rm d}t}{t}\right)^{1/q}<+\infty.    
\end{align*}
Notice that the $L^{p}$-spaces on $M$ can be defined by
\[
\|L\|_{L^{p}(M)}:=\|L\|_{L^{p,p}(M)}=\left(\int_0^{+\infty}\mu_t(L)^p{\rm d}t\right)^{1/p}.
\]
In the case $q=\infty$, $L^{p,\infty}(M)$ is the set of all operators $L\in S(M)$ such that 
\[
\|L\|_{L^{p,\infty}(M)}=\sup_{t>0}t^{1/p}\mu_t(L).
\]
\end{defn}
As announced in the introduction, we are interested in the Lorentz norms of spectral multipliers of the form $\phi(|L|)$, where $L$ is an affiliated operator. A minor technical point requires attention: we want to allow the function $\phi$ to be complex-valued. For this reason, we state the following lemma concerning the polar decomposition of such spectral multipliers. This result is fundamental for the considerations introduced in Definition~\ref{generalTValues}.

\begin{lem}\label{AbsValSpecMult}
   Let $L$ be a closed (possibly unbounded) operator affiliated with a semifinite von Neumann algebra $M$, and let $\phi : [0,+\infty) \to \mathbb{C}$ be a Borel measurable function. Then the positive part of the operator $\phi(|L|)$ appearing in its polar decomposition is obtained via the Borel functional calculus as
\[
|\phi(|L|)|=\widetilde{\phi}(|L|),
\]
where $\widetilde{\phi}$ is the function defined by $\widetilde{\phi}(s)=|\phi(s)| = (|\cdot|\circ \phi)(s)$.
\end{lem}
\begin{proof}
First, note that if $T$ is an operator affiliated with a semifinite von Neumann algebra $M$, then its positive part is defined by $|T| = \sqrt{T^*T},$ where $T^*$ denotes the adjoint of $T$. In general, when dealing with products of spectral multipliers associated with different functions, one must carefully consider domain issues. In the present setting, however, this difficulty does not arise, since we work directly with a spectral multiplier and its adjoint. Indeed, by \cite[Theorem 13.24 (5), p.~362]{rudin}, we obtain the identities
\[
\phi(|L|)^*\,\phi(|L|) \;=\; \phi(|L|)\,\phi(|L|)^* \;=\; |\phi|^2(|L|).
\]

Therefore, the positive part of the spectral multiplier $\phi(|L|)$, which we temporarily denote by $P$, satisfies
\begin{align*}
P
    &= \sqrt{[\phi(|L|)]^*\phi(|L|)}= \sqrt{|\phi|^2(|L|)}.
\end{align*}
Moreover, the Borel functional calculus is compatible with the composition of functions (see \cite[Corollary 5.6.29, p.~363]{composition}), hence
\[
P = \widetilde{\phi}(|L|) = |\phi(|L|)|,
\]
where $\widetilde{\phi}(s)=|\phi(s)| = (|\cdot|\circ \phi)(s)$. This completes the proof.    
\end{proof}

\section{Main result and some applications}\label{main-aplications}

In this section, we prove the main result of the paper (Theorem~\ref{additional}) and present several applications to well-known integral equations. More precisely, we show that the function defining the spectral multiplier need not itself satisfy the decay conditions. It is sufficient that the modulus of this function be bounded by another function that does satisfy the required assumptions.

\begin{proof}[Proof of Theorem \ref{additional}]
By the definition of Lorentz spaces and \cite[Proposition 2.9]{RR2020}, we have
\begin{equation}
\label{lorentz def}
\|\phi(|L|)\|_{L^{r,\infty}(M)}
    = \sup_{s>0} s \left[\tau\!\left(E_{(s,+\infty)}\big(|\phi(|L|)|\big)\right)\right]^{\frac{1}{r}},
\end{equation}
where $|\phi(|L|)|$ denotes the positive part of the operator $\phi(|L|)$ arising from its polar decomposition. By Lemma~\ref{AbsValSpecMult}, this positive part coincides with the spectral multiplier $\widetilde{\phi}(|L|)$ associated with the function $\widetilde{\phi}(s)=|\phi(s)|.$ The Borel functional calculus \cite{BorelFunctional} allows us to express the spectral projections as follows
\[
E_{(s,+\infty)}(|L|) = \int_0^{+\infty} \chi_{(s,+\infty)}(t)\,{\rm d}E_t(|L|)= \chi_{(s,+\infty)}(|L|),
\]
therefore, by \cite[Corollary 5.6.29, p. 363]{composition}
\[
E_{(s,+\infty)}(\Tilde{\phi}(|L|)) = \chi_{(s,+\infty)}(\Tilde{\phi}(|L|)) = (\chi_{(s,+\infty)}\circ \Tilde{\phi})(|L|) = \int_0^{+\infty} \chi_{(s,+\infty)}(\Tilde{\phi}(t))\,{\rm d}E_t(|L|). 
\]
Using \cite[Proposition 2.6]{RR2020} (see also \cite{pacific}) we are able to calculate the trace of this operator, namely
\[
\tau(E_{(s,+\infty)}(\Tilde{\phi}(|L|))) = \int_0^{+\infty} \chi_{(s,+\infty)}(\Tilde{\phi}(t))\,{\rm d}\mu_t,
\]
where $\mu_t = \tau(E_t(|L|))$ is a well-defined measure. Since $\Tilde{\phi}(v) =|\phi(v)|\leqslant \psi(v)$ for all $v\in[0,+\infty)$, we have that $\chi_{(s,+\infty)}(\Tilde{\phi}(v))\leqslant \chi_{(s,+\infty)}(\psi(v))$ for all $v\in[0,+\infty)$. As a consequence of the monotonicity of $\mu_t$ we get 
\begin{align}
\label{monotonicity}
\tau(E_{(s,+\infty)}(\Tilde{\phi}(|L|)))&\leqslant \int_0^{+\infty} \chi_{(s,+\infty)}(\psi(t))\,{\rm d}\mu_t =\tau(E_{(s,+\infty)}(\psi(|L |))).
\end{align}
Putting \eqref{lorentz def} and \eqref{monotonicity} together we obtain
\[
\|\phi(|L|)\|_{L^{r, \infty}(M)} \leqslant \sup_{s>0}s\left[\tau(E_{(s,+\infty)}(\psi(|L|)))\right]^{\frac{1}{r}}.
\]
From the hypothesis given on $\psi$, the proof can be completed by mimicking the proof of \cite[Theorem 6.1]{RR2020}.
\end{proof}

\begin{rem}
The condition $\psi(0)=1$ in Theorem \ref{additional} can be rescaled, so we can prove the same result if one has $\psi(0)=\kappa_0>0.$ 
\end{rem}

In the next subsections, we establish for heat, wave and Schr\"odinger type equations the $L^p-L^q$ norm estimates for $1<p\leqslant 2\leqslant q<+\infty$. Indeed, we will see that the norm estimates can be reduced to the time asymptotic of the propagator (some of them oscillatory) in the noncommutative Lorentz space norm. Also, we show that the latter norm mainly involves calculating the trace of the spectral projections of the operator $\mathscr{L}$.

\medskip In each equation below we assume that $G$ is a locally compact separable unimodular group, $\mathscr{L}$ is any positive linear left invariant operator on $G$ (maybe unbounded) and $^{C}\partial_{t}^{\beta}$ is the Dzhrbashyan-Caputo fractional derivative of order $\beta$ from \eqref{caputo-alternative-uso}.

\subsection{\texorpdfstring{$\mathscr{L}$}{L}-heat type equation} We consider the following heat type equation: 
\begin{equation}\label{heatlocally}
\begin{split}
^{C}\partial_{t}^{\beta}w(t,x)+\mathscr{L}w(t,x)&=0, \quad t>0,\,\, x\in G,\,\,0<\beta<1, \\
w(t,x)|_{_{_{t=0}}}&=w_0(x).
\end{split}
\end{equation}
Our next result also covers the case $\beta=1$, i.e. the classical $\mathscr{L}$-heat equation, but it is already known \cite[Section 7]{RR2020}. 

\begin{thm}\label{heatlocally-thm}
Let $G$ be a locally compact separable unimodular group, $0<\beta<1$ and $1<p\leqslant 2\leqslant q<+\infty$. Let $\mathscr{L}$ be any positive left invariant operator on $G$ (maybe unbounded) such that
\begin{equation}\label{need}
\sup_{t>0}\sup_{s>0}\big[\tau\big(E_{(0,s)}(\mathscr{L})\big)\big]^{\frac{1}{p}-\frac{1}{q}}E_\beta(-t^{\beta}s)<+\infty.   \end{equation}
If $w_0\in L^p(G)$ then there exists a unique solution $w\in \mathcal{C}\big((0,+\infty) ;L^q(G)\big) $ for the $\mathscr{L}$-heat type equation \eqref{heatlocally} represented by
\[
w(t,x)=E_\beta(-t^{\beta}\mathscr{L})w_0(x),\quad t>0,\,\,x\in G.
\]
In the particular case that  
\begin{equation}\label{tracecondition}
\tau\big(E_{(0,s)}(\mathscr{L})\big)\lesssim s^{\lambda},\quad s\to+\infty,\quad\text{for some}\quad \lambda>0,
\end{equation}
we get the following time decay rate for the solution of equation \eqref{heatlocally} for all $t>0,$
\[
\|w(t,\cdot)\|_{L^q(G)}\leqslant C_{\beta  ,\lambda,p,q}t^{-\beta\lambda\left(\frac{1}{p}-\frac{1}{q}\right)}\|w_0\|_{L^p(G)},\quad \frac{1}{\lambda}\geqslant\frac{1}{p}-\frac{1}{q}.    
\]
\end{thm}
\begin{proof}
Since the operator $\mathscr{L}$ is positive we have that the solution operator of equation \eqref{heatlocally} is given by $w(t,x)=E_\beta(-t^\beta \mathscr{L})w_0(x)$ \cite[Chapter 3]{thesis} (see also \cite[Prop. 3.8 and Def. 2.3]{thesis2001}). So, by \cite[Theorem 5.1]{RR2020}, we have
\begin{equation}\label{previous-10}
\|w(t,\cdot)\|_{L^q(G)}=\|E_\beta(-t^{\beta}\mathscr{L})w_0\|_{L^q(G)}\lesssim \|E_\beta(-t^\beta \mathscr{L})\|_{L^{r,\infty}(VN_R(G))}\|w_0\|_{L^p(G)},
\end{equation}
with $\frac{1}{r}=\frac{1}{p}
-\frac{1}{q}$, where the above Lorentzian norm is calculated as \cite[Theorem 6.1]{RR2020}:
\[
\|E_\beta(-t^\beta \mathscr{L})\|_{L^{r,\infty}(VN_R(G))}=\sup_{s>0}\big[\tau\big(E_{(0,s)}(\mathscr{L})\big)\big]^{\frac{1}{r}}E_\beta(-t^{\beta}s).
\]
In the above result, we are using that $E_\beta(-t^{\beta}s)$ $(t,s>0)$ is completely monotonic for all $0<\beta\leqslant1$ (\cite{Pollard}) such that $E_\beta(0)=1$ and $\lim_{s\to+\infty}E_\beta(-t^{\beta}s)=0$ by the uniform estimate \eqref{uniform-estimate}. Thus, by \eqref{tracecondition} and again estimate \eqref{uniform-estimate} we get
\[
\|E_\beta(-t^\beta \mathscr{L})\|_{L^{r,\infty}(VN_R(G))}\lesssim \sup_{s>0}\frac{s^{\frac{\lambda}{r}}}{1+t^\beta s}. 
\]
Notice first that for $\lambda/r=1$, the supremum is bounded by $t^{-\beta}$. Also, the above supremum is attained at $s=\frac{\lambda t^{-\beta}}{r-\lambda}$ whenever $\frac{1}{\lambda}>\frac{1}{p}-\frac{1}{q}$, hence
\begin{equation}\label{asymtotic-heat}
\|E_\beta(-t^\beta \mathscr{L})\|_{L^{r,\infty}(VN_R(G))}\leqslant C_{\beta,\lambda,p,q}t^{-\beta \lambda/r}, 
\end{equation}
and then the result follows by \eqref{previous-10}.
\end{proof}

\subsection{\texorpdfstring{$\mathscr{L}$}{L}-wave type equation}\label{wave-locally}

Here we investigate the solution of the following equation, which in a sense interpolates between wave (without being wave, $\beta<2$) and heat types:  
\begin{equation}\label{locallywave}
\begin{split}
^{C}\partial_{t}^{\beta}w(t,x)+\mathscr{L}w(t,x)&=0, \quad t>0,\,\, x\in G,\,\,1<\beta<2, \\
w(t,x)|_{_{_{t=0}}}&=w_0(x), \\
\partial_t w(t,x)|_{_{_{t=0}}}&=w_1(x).
\end{split}
\end{equation}

\begin{thm}\label{locally-wave-thm}
Let $G$ be a locally compact separable unimodular group, $1<\beta<2$ and $1<p\leqslant 2\leqslant q<+\infty$. Let $\mathscr{L}$ be any positive left invariant operator on $G$ (maybe unbounded) satisfying the following condition 
\begin{equation}\label{need-2}
\sup_{t>0}\sup_{s>0}\frac{\big[\tau\big(E_{(0,s)}(\mathscr{L})\big)\big]^{\frac{1}{p}-\frac{1}{q}}}{1+t^\beta s}<+\infty.   
\end{equation}
If $w_0,w_1\in L^p(G)$ then there exists a unique solution $w\in\mathcal{C}\big((0,+\infty);L^q(G)\big)$ for the $\mathscr{L}$-wave type equation \eqref{locallywave} given by
\[
w(t,x)=E_\beta(-t^{\beta}\mathscr{L})w_0(x)+\prescript{RL}{0}I^{1}_t E_{\beta}(-t^{\beta}\mathscr{L})w_1(x),\quad t>0,\,\,x\in G.
\]
In particular, if the condition \eqref{tracecondition} holds then for any $1<p\leqslant 2\leqslant q<+\infty$ such that $\frac{1}{\lambda}\geqslant\frac{1}{p}-\frac{1}{q}$ we get the following time decay rate for the solution of equation \eqref{locallywave}: 
\[
\|w(t,\cdot)\|_{L^q(G)}\leqslant  C_{\beta,\lambda,p,q}t^{-\beta\lambda\left(\frac{1}{p}-\frac{1}{q}\right)} \big(\|w_0\|_{L^p(G)}+t\|w_1\|_{L^p(G)}\big).    
\]
\end{thm}
\begin{proof}
We know that the solution operator of equation \eqref{locallywave} can be expressed as (see Subsection \ref{fde}): 
\begin{equation}\label{wave-eq}
w(t,x)=E_\beta(-t^{\beta}\mathscr{L})w_0(x)+\prescript{RL}{0}I^{1}_t E_{\beta}(-t^{\beta}\mathscr{L})w_1(x).
\end{equation}
Notice now that 
\begin{equation}\label{integralmittag}\prescript{RL}{0}I^{1}_t E_{\beta}(-t^{\beta}s)=\sum_{k=0}^{+\infty}\frac{(-s)^{k}}{\Gamma(\beta k+1)}\frac{t^{\beta k+1}}{\beta k+1}=t\sum_{k=0}^{+\infty}\frac{(-st^{\beta})^{k}}{\Gamma(\beta k+2)}=tE_{\beta,2}(-t^{\beta}s). 
\end{equation}
Taking the propagators of equation \eqref{wave-eq},  equality \eqref{integralmittag}, the condition \eqref{need-2}, estimate \eqref{uniform-estimate}, Theorem \ref{additional} and \cite[Theorem 5.1]{RR2020}, we obtain 
\begin{align*}
&\|w(t,\cdot)\|_{L^q(G)}\leqslant \|E_\beta(-t^{\beta}\mathscr{L})w_0\|_{L^q(G)}+\|\prescript{RL}{0}I^{1}_t E_{\beta}(-t^{\beta}\mathscr{L})w_1\|_{L^q(G)} \\
&\lesssim \sup_{v>0}\frac{1}{1+vt^{\beta}}\big(\tau(E_{(0,v)}(\mathscr{L}))\big)^{\frac{1}{p}-\frac{1}{q}} \big(\|w_0\|_{L^p(G)}+t\|w_1\|_{L^p(G)}\big) \\
&\lesssim \sup_{v>0}\frac{v^{\lambda\left(\frac{1}{p}-\frac{1}{q}\right)}}{1+vt^{\beta}}\big(\|w_0\|_{L^p(G)}+t\|w_1\|_{L^p(G)}\big)\lesssim t^{-\beta \lambda\left(\frac{1}{p}-\frac{1}{q}\right)}\big(\|w_0\|_{L^p(G)}+t\|w_1\|_{L^p(G)}\big),
\end{align*}
completing the proof.
\end{proof}

\subsection{\texorpdfstring{$\mathscr{L}$}{L}-Schr\"odinger type equation} Consider the $\mathscr{L}$-Schr\"odinger type equation: 
\begin{align}\label{asterisco}
\begin{split}
i\,^{C}\partial_{t}^{\beta}w(t,x)+\mathscr{L}w(t,x)&=0, \quad t>0,\,\, x\in G,\,\,0<\beta<1, \\
w(t,x)|_{_{_{t=0}}}&=w_0(x).
\end{split}
\end{align}
  
\begin{thm}\label{Schrodinger-locally-thm}
Let $0<\beta<1$ and $1<p\leqslant 2\leqslant q<+\infty$. Suppose that the operator $\mathscr{L}$ satisfies the condition \eqref{need-2} with $0<\beta<1.$ If $w_0\in L^p(G)$ then there exists a unique solution $w\in\mathcal{C}\big((0,+\infty) ;L^q(G)\big)$ for the $\mathscr{L}$-Schr\"odinger type equation \eqref{asterisco} represented by
\[
w(t,x)=E_\beta(it^{\beta}\mathscr{L})w_0(x),\quad t>0,\,\,x\in G.
\]
In particular, if the condition \eqref{tracecondition} holds then for any $1<p\leqslant 2\leqslant q<+\infty$ such that $\frac{1}{\lambda}\geqslant\frac{1}{p}-\frac{1}{q}$ we get the following time decay rate for the solution of equation \eqref{asterisco} for all $t>0,$
\[
\|w(t,\cdot)\|_{L^q(G)}\leqslant C_{\beta  ,\lambda,p,q}t^{-\beta\lambda\left(\frac{1}{p}-\frac{1}{q}\right)}\|w_0\|_{L^p(G)}.    
\]
\end{thm}

\begin{proof}
Notice that the solution operator of equation \eqref{asterisco} is given by
\begin{equation*}
    w(t,x)=E_\beta(it^\beta \mathscr{L})w_0(x),\quad 0<\beta<1.
\end{equation*}  
Hence, by \cite[Theorem 5.1]{RR2020}  we have
\begin{equation}\label{previous-1}
\|w(t,\cdot)\|_{L^q(G)}=\|E_\beta(it^{\beta}\mathscr{L})w_0\|_{L^q(G)}\lesssim \|E_\beta(it^\beta \mathscr{L})\|_{L^{r,\infty}(VN_R(G))}\|w_0\|_{L^p(G)},
\end{equation}
with $\frac{1}{r}=\frac{1}{p}
-\frac{1}{q}$. Since the argument of the Mittag-Leffler function is a pure imaginary number then we can apply inequality \eqref{uniform-estimate} and get
\[
|E_\beta(it^{\beta}s)|\leqslant \frac{C}{1+t^{\beta}s},\quad s,t\geqslant0,\quad 0<\beta<1. 
\]
In this way, applying Theorem \ref{additional} with  $\psi(s)=C/(1+t^{\beta}s)$ we finally obtain that
\[
\|E_\beta(it^\beta \mathscr{L})\|_{L^{r,\infty}(VN_R(G))}\lesssim \sup_{s>0}\frac{s^{\frac{\lambda}{r}}}{1+t^{\beta}s}\leqslant  C_{\beta,\lambda,p,q}t^{-\beta \lambda/r},\quad \frac{1}{\lambda}\geqslant\frac{1}{p}-\frac{1}{q},
\]
and thus the result follows by \eqref{previous-1}.   
\end{proof}

\begin{rem}
We emphasize that Theorems \ref{locally-wave-thm}, \ref{Schrodinger-locally-thm} do not cover the classical $\mathscr{L}$-wave equation ($\beta=2$) and $\mathscr{L}$-Schr\"odinger equation ($\beta=1$), respectively. For this case one would need a more delicate analysis of the problem, especially focusing on understanding the behavior of the trace of the spectral projection of the Schr\"odinger (wave) propagator $\tau(E_{(\lambda,+\infty)}(e^{it\mathscr{L}}))$. The latter remains an open problem.  
\end{rem}

\section{Examples}\label{examples-section}
In this section we provide some explicit examples to clarify and compare the wide class of groups and operators we can be treated with our results. We do it for the case of heat and wave type equations, but obviously the same examples work for the Schr\"odinger type equations.

Let us consider the following general $\mathfrak{D}$-heat-wave type equation: 
\begin{equation}\label{examples}
\begin{split}
^{C}\partial_{t}^{\alpha}w(t,x)-\mathfrak{D}w(t,x)&=0, \quad t>0,\,\, x\in G, \\
w(t,x)|_{_{_{t=0}}}&=w_0(x),\quad w_0\in L^p(G),  \\
\partial_t w(t,x)|_{_{_{t=0}}}&=w_1(x),\quad w_1\in L^p(G),
\end{split}
\end{equation}
where $0<\alpha<2$ $(\alpha\neq1)$, $1<p\leqslant 2,$ $\mathfrak{D}$ is any positive linear left invariant operator on a locally compact separable unimodular group $G$. We use the first initial condition $w_0$ when $0<\alpha<1$. We consider both initial conditions $w_0,w_1$ when $1<\alpha<2.$  

By Theorem \ref{locally-wave-thm}, we get for any $1<p\leqslant 2\leqslant q<+\infty$ the following time decay rate for the solution of equation \eqref{examples}: 
\begin{equation}\label{solu-2}
\|w(t,\cdot)\|_{L^q(G)}\leqslant C_{\alpha,\nu,Q,p,q}t^{-\alpha\lambda\left(\frac{1}{p}-\frac{1}{q}\right)} \big(\|w_0\|_{L^p(G)}+t\|w_1\|_{L^p(G)}\big),  
\end{equation}
whenever 
\[
\tau\big(E_{(0,s)}(\mathfrak{D})\big)\lesssim s^{\lambda},\quad s\to+\infty,\quad\text{for some}\quad \lambda>0,\quad \frac{1}{\lambda}\geqslant\frac{1}{p}-\frac{1}{q}.
\]
Below we give several examples of equation \eqref{examples} in different {\it Lie} groups and particular operators acting there, whose trace of the spectral projections behave like $s^{\lambda}$ as $s\to+\infty$. Moreover, in some cases, we can claim the sharpness of the time-decay rate.  
\vspace{0.3cm}

\begin{center}
\begin{tabular}{ |c| c | c | c | c |}
    \hline
    \# & \textbf{Group} & \textbf{Operator} $\mathbf{\mathfrak{D}}$ & $\mathbf{\tau\left(E_{(0, s)}(\mathfrak{D})\right)}$ & \textbf{Time Decay} \\ \hline
     1 & The Euclidean space $\R^n$ & $\Delta_{\R^n}$ & $s^{n/2}$ & $t^{-\alpha n/2 \left(\frac{1}{p}-\frac{1}{q}\right)}$   \\ \hline
     2 & Any compact Lie group $G$ & $\Delta_{sub}$ & $s^{Q/2}$ & $t^{-\alpha Q/2 \left(\frac{1}{p}-\frac{1}{q}\right)}$   \\ \hline
     3 & The Heisenberg group $\mathbb{H}^n$&  $\mathscr{L}$ & $s^{n+1}$ & $t^{-\alpha (n+1) \left(\frac{1}{p}-\frac{1}{q}\right)}$  \\ \hline
     4 & Any graded Lie group $G$ &  $\mathcal{R}$ & $s^{Q/\nu}$ & $t^{-\alpha Q/\nu \left(\frac{1}{p}-\frac{1}{q}\right)}$ \\ \hline
     5 & Engel group $\mathfrak{B}_4$ &  $\mathfrak{D_1}$ & $s^{3}$ & $t^{-3\alpha\left(\frac{1}{p}-\frac{1}{q}\right)}$ \\ \hline
     6 & Cartan group $\mathfrak{B}_5$ &  $\mathfrak{D_2}$ & $s^{9/2}$ & $t^{-9\alpha/2 \left(\frac{1}{p}-\frac{1}{q}\right)}$ \\ \hline
     7 & Any connected Lie group $G$ &  $\mathcal{L}$ & $s^{Q_*/m}$ & $t^{-\alpha Q_*/m \left(\frac{1}{p}-\frac{1}{q}\right)}$ \\ \hline
    \end{tabular}
\end{center}
\vspace{0.3cm}

\begin{itemize}
\item[1.] The Laplacian $\Delta_{\R^n}$ on the Euclidean space $\R^n$.  By \cite[Example 7.3]{RR2020}, it follows that
\[
\tau\big(E_{(0,s)}(\Delta_{\R^n})\big)\lesssim s^{n/2},\quad s\to+\infty.
\]
Here we can recover the sharp estimate (time-decay) given in \cite[Theorem 3.3, item (i)]{uno2} whenever $\frac{2}{n}>\frac{1}{p}-\frac{1}{q}$. 

    \item[2.] The sub-Laplacian $\Delta_{sub}$ on a compact Lie group. By \cite{[35]}, we have
\[
\tau\big(E_{(0,s)}(-\Delta_{sub})\big)\lesssim s^{Q/2},\quad s\to+\infty,
\]
where $Q$ is the Hausdorff dimension of $G.$ 

\item[3.] The positive sub-Laplacian on the Heisenberg group $\mathbb{H}^n$. By \cite[Formula (7.17)]{RR2020}, we get
\[
\tau\big(E_{(0,s)}(\mathscr{L})\big)\lesssim s^{n+1},\quad s\to+\infty.
\]

\item[4.] A positive Rockland operator $\mathcal{R}$ on a graded Lie group. By \cite[Theorem 8.2]{david}, we obtain
\[
\tau\big(E_{(0,s)}(\mathcal{R})\big)\lesssim s^{Q/\nu},\quad s\to+\infty,
\]
where $Q$ is the homogeneous dimension of $G$. In this case the estimate is sharp and we can then claim the sharpness of the time-decay rate.

\item[5.] The non-Rockland-type operator $\mathfrak{D_1} = -(X_1^2 + X_2^2 + X_3^2 +X_4^2 +X_4^{-2})$ on the Engel group $\mathfrak{B}_4$, where $\{X_i\}$ are the vector fields that form the canonical basis of its Lie algebra. By \cite[Example 2.2]{Marianna}, we know 
\[
\tau\big(E_{(0,s)}(\mathfrak{D_1})\big) \lesssim s^{3},\quad s\to+\infty.
\]

\item[6.] The non-Rockland-type operator $\mathfrak{D_2} = -(X_1^2 + X_2^2 + X_3^2 +X_4^2 + X_5^2+X_4^{-2}+X_5^{-2})$ on the Cartan group $\mathfrak{B}_5$, where $\{X_i\}$ are the vector fields that form the canonical basis of its Lie algebra. By \cite[Example 3.2]{Marianna}, one has
\[
\tau\big(E_{(0,s)}(\mathfrak{D_1})\big) \lesssim s^{9/2},\quad s\to+\infty.
\]

\item[7.] An $m$-th order weighted subcoercive positive operator on a connected unimodular Lie group. By \cite[Proposition 0.3]{david2}, we arrive at
$$
\tau\left(E_{(0,s)}(\mathcal{L})\right) \lesssim s^{\frac{Q_*}{m}}, \quad s \rightarrow \infty,
$$
where $Q_*$ is the local dimension of $G$ relative to the chosen weighted structure on its Lie algebra.
\end{itemize}
{\bf Example of an abelian locally compact group of non-Lie type.} We conclude this section by presenting in more detail the case of the group of $\rho$-adic numbers, denoted by $\Q_{\rho}$ ($\rho$ is a prime number). Since there is no smooth structure in $\Q_{\rho}$ we can not define a Laplacian-type operator using derivatives, but we can do it as a Fourier multiplier. In fact, we consider the Vladimirov operator \cite{[30]}, an analogous operator to the fractional Laplacian, defined as follows: 
\begin{equation}
    \mathfrak{D}^\mu f(x) = \int_{\Q_{\rho}} e^{2\pi i \{x\xi\}_\rho} |\xi|_\rho^\mu \widehat{f}(\xi)\,{\rm d}\xi,\quad \mu>0,\,\,x\in \Q_\rho,
\end{equation}
where $\widehat{f}$ is the Fourier transform of $f$, $|\cdot|_\rho$ is the $\rho$-adic norm and $\{\cdot\}_\rho$ is the fractional part of a $\rho$-adic number. This means that $\widehat{\mathfrak{D}^\mu f}(\xi) = |\xi|_\rho^\mu\widehat{f}(\xi),$ so that its symbol is given by $\sigma_{\mathfrak{D}^\mu}(\xi)=|\xi|_\rho^\mu$. From the general theory (see \cite[Theorem 5.6.26, p. 360]{composition} and \cite[Corollary 5.6.29, p. 363]{composition}) we get that $E_{(0,s)}\left(\mathfrak{D}^\mu\right) = \chi_{(0,s)}(|\sigma_{\mathfrak{D}^\mu}(\xi)|).$ Therefore the trace of this spectral projection can be directly computed as 
\[
\tau\left(E_{(0,s)}\left(\mathfrak{D}^\mu\right)\right) = \int\limits_{\substack{\Q_\rho \\ |\sigma_{\mathfrak{D}^\mu}(\xi)| \leqslant s}} \,{\rm d}\xi = \int\limits_{\substack{\Q_\rho \\ |\xi|_\rho^\mu \leqslant s}} \,{\rm d}\xi = \operatorname{Vol}(B_{s^{1/\mu}}(0)), 
\]
where as usual $B_r(y)$ denotes the $\rho$-adic ball of radius $r$ and center $y$. Let $v$ be the smallest integer such that $1/\rho^v\sim s^{1/\mu}$, thus by the topology given in $\Q_\rho$ we have
\[
\operatorname{Vol}(B_{s^{1/\mu}}(0)) = \operatorname{Vol}(B_{\rho^{-v}}(0)) = 1/\rho^v,
\]
since this is the normalization of the Haar measure in $\Q_\rho$ (\cite{burba}). Hence 
\[
\tau\left(E_{(0,s)}\left(\mathfrak{D}^\mu\right)\right) = 1/\rho^v \lesssim s^{1/\mu},\quad s \rightarrow \infty.
\]
Finally, we obtain time decay rate $t^{-\alpha/\mu \left(\frac{1}{p}-\frac{1}{q}\right)}$ for the solution of $\mathfrak{D}^\mu$-equation \eqref{examples}, with $\lambda=1/\mu$ in \eqref{solu-2}.

\section{Acknowledgements}
The authors were supported by the FWO Odysseus 1 grant G.0H94.18N: Analysis and Partial Differential Equations, the Methusalem programme of the Ghent University Special Research Fund (BOF) (Grant number 01M01021). MR is also supported by EPSRC grant EP/R003025/2 and FWO Senior Research Grant G011522N.

\section*{Data availability} No new data was collected or generated during the course of this research.

\section*{Declarations} 
{\bf Conflict of interest.} The authors declare no competing interests.

\end{document}